\documentclass[11pt]{article}

\usepackage{fullpage}
\usepackage{amsmath, amsthm, amsfonts, amssymb, amstext, mathrsfs, enumerate}
\usepackage{graphicx, ragged2e, lscape, framed, xcolor}
\usepackage{subfiles}
\usepackage{microtype}

\theoremstyle{plain}
\newtheorem{theorem}{Theorem}[section]
\newtheorem{lemma}[theorem]{Lemma}

\newtheorem{claim}{Claim}[section]

\numberwithin{equation}{section}
\allowdisplaybreaks

\newcommand{\affl}[3]{%
    \par\noindent #1\\*
    {\small Email: {\ttfamily #2}}\\*
    \textsc{#3}\par\smallskip
}

\usepackage[pagebackref]{hyperref}
\hypersetup{
    colorlinks=true,
    urlcolor=blue,
    linkcolor=purple,
    citecolor=blue,
    pdftitle={Supersaturation: Cliques and Positive Square Energy},
    pdfauthor={Abhay Jayarajan, M. Rajesh Kannan, Shivaramakrishna Pragada},
}

\DeclareMathOperator{\tr}{tr}

\newcommand{\one}{\mathbf 1}

\newcommand{\norm}[1]{\left\lVert #1\right\rVert}
\newcommand{\ip}[2]{\left\langle #1,#2\right\rangle}

\newcommand{\cP}{\mathcal{P}}

\title{\textbf{Supersaturation: Cliques and Positive Square Energy}}
\author{Abhay Jayarajan, M. Rajesh Kannan, Shivaramakrishna Pragada}
\date{\today}

\begin{document}
\maketitle

\begin{abstract}

Let $G$ be a simple graph of order $n$ and let $k_r$ denote the number of unlabeled copies of $K_r$ in $G$. Let $s^+$ denote the sum of the squares of the positive adjacency eigenvalues of $G$. For every fixed integer $r\geq2$, we prove that, as $n \to \infty$
\[
    \sqrt{s^+}>\left(1-\frac1r\right)n
    \quad\Longrightarrow\quad
    k_{r+1}\geq
    \left(\frac1{r^{r-1}}-o_r(1)\right)n^{r-1}.
\]
The coefficient $r^{-(r-1)}$ is asymptotically best possible; in particular, the triangle case gives $k_3\geq n/2-o(n)$. The proof combines positive square energy stability, the clique removal lemma, and a sharp clique count around a maximum $r$-cut. The triangle case requires an additional block-matrix argument.
\end{abstract}



\section{Introduction}\label{sec:introduction}

\subsection{Background}
Extremal graph theory asks how large a graph can be while avoiding a prescribed subgraph. Supersaturation asks what happens immediately after this extremal threshold is crossed: how many copies of the forbidden graph must then appear? Even the addition of a single edge can force a number of copies growing polynomially with the order of the graph. A central problem is to identify both the order of growth of this increment and its sharp leading coefficient. In this paper, we study this question when the extremal threshold is measured by the combined contribution of all positive adjacency eigenvalues.

Throughout, $G$ is a finite simple graph of order $n$ and size $m$. Let $V(G)$ and $E(G)$ denote the vertex and edge sets of $G$ respectively. The adjacency matrix of $G$ is denoted by $A=A(G)$, and its adjacency eigenvalues are ordered as
\[
    \lambda_1(G)\geq\lambda_2(G)\geq\cdots\geq\lambda_n(G).
\]
For an integer $j\ge2$, let $k_j(G)$ denote the number of unlabeled copies of $K_j$ in $G$. 
We write
\[
    s^+(G)=\sum_{\lambda_i(G)>0}\lambda_i^2(G),
    \qquad
    s^-(G)=\sum_{\lambda_i(G)<0}\lambda_i^2(G).
\]
Thus $s^+(G)+s^-(G)=2m$.
 
Let $T_r(n)$ be the complete $r$-partite graph of order $n$ whose part sizes differ by at most one, and let $t_r(n)=|E(T_r(n))|$. For a vertex $v$ let $N_G(v)$ denote the set of neighbors of $v$ in $G$. Write $d_G(v)= |N_G(v)|$. For disjoint vertex sets $U,W\subseteq V(G)$, let $E_G(U,W)$ be the set of edges of $G$ between $U$ and $W$, and put $e_G(U,W)=|E_G(U,W)|$. We write $d_G(v,U)=|N_G(v)\cap U|$. When the underlying graph is clear from context, we omit it from the notation and write, for example, $k_j$, $s^+$, $s^-$, $N(v)$, $d(v)$, and $e(U,W)$. Unless otherwise stated, asymptotic statements are understood
along sequences of graphs whose orders tend to infinity, with $r$ and any specified positive constants fixed.

\textbf{Background on edge supersaturation: }The classical starting point in extremal graph theory is the theorem of Mantel~\cite{Mantel1907} and its extension by Tur\'an~\cite{Turan1941}: every $K_{r+1}$-free graph on $n$ vertices has at most $t_r(n)$ edges, with equality precisely for $T_r(n)$. Rademacher's theorem, recorded by Erd\H{o}s~\cite{Erdos1955}, shows that one edge above Mantel's threshold forces at least $\lfloor n/2\rfloor$ triangles. Subsequent work of Erd\H{o}s~\cite{Erdos1962Rademacher,Erdos1969Cliques} extended this phenomenon to more added edges and to larger cliques. 
Lov\'asz and Simonovits~\cite{LovaszSimonovits1983} determined the minimum number of $(r+1)$-cliques for a wide range of edge counts just above $t_r(n)$, showing that this minimum grows in proportion
to the number of extra edges.

Erd\H{o}s and Simonovits~\cite{ErdosSimonovits1983} proved a general supersaturation theorem, showing that having a fixed positive multiple of $n^2$ additional edges beyond the maximum allowed in an $F$-free graph forces many copies of $F$. For cliques, the problem of asymptotically determining the smallest possible number of copies for given numbers of vertices and edges led to the results of Razborov~\cite{Razborov2008Triangles} for triangles and Nikiforov~\cite{Nikiforov2011Cliques} for triangles and $4$-cliques, and culminated in Reiher's theorem for cliques of arbitrary size~\cite{Reiher2016CliqueDensity}. Closer to the extremal threshold, Mubayi~\cite{Mubayi2010} established a counting theorem for every color-critical graph $F$, that is, a graph containing an edge whose deletion lowers its chromatic number. Pikhurko and Yilma~\cite{PikhurkoYilma2017Supersaturation} refined this theory and studied the structure of minimizing graphs for color-critical graphs. More recent work of Ma and Yuan~\cite{MaYuan2025Supersaturation} extends the investigation beyond the color-critical setting. Zhang, Liu, Hong, and Hu~\cite{ZhangLiuHongHu2026} established analogous supersaturation results for certain uniform hypergraph expansions of color-critical graphs.

\textbf{Background on spectral (radius) supersaturation: }Spectral supersaturation replaces an edge-count assumption by a condition on the adjacency spectral radius $\lambda_1(G)$. Wilf's inequality~\cite{Wilf1986} gives
\[
  \lambda_1(G)\le \left(1-\frac 1r \right) n
  \qquad\text{whenever $G$ is $K_{r+1}$-free}.
\]
Since $\lambda_1(G)\ge2m/n$, spectral conditions can apply to graphs whose edge count alone does not cross the corresponding extremal threshold. Bollob\'as and Nikiforov~\cite[Theorem~2]{BollobasNikiforov2007Cliques} obtained the quantitative bound
\[
 k_{r+1}(G)\ge
 \left(\frac{\lambda_1(G)}{n}-\left(1-\frac 1r \right)\right)
 \frac{r(r-1)}{r+1}\left(\frac{n}{r}\right)^{r+1}
\]
for $r \geq 2$. This guarantees at least a fixed positive multiple of $n^{r+1}$ copies of $K_{r+1}$ when the spectral radius exceeds $\left(1-\frac1r\right)n$ by a fixed positive multiple of $n$. However, when the spectral radius is only just above the threshold, this bound alone need not guarantee as many cliques as adding one edge to a Turán graph would create. Nikiforov's spectral Erd\H{o}s--Stone-Bollob\'as theorem~\cite{Nikiforov2009SpectralESB} and spectral stability theorem~\cite{Nikiforov2009Stability} provide further connections between spectral excess, large multipartite subgraphs, and proximity to a Tur\'an graph.

The problem of determining how many triangles must occur when the
threshold is exceeded only slightly has received particular attention. Ning and Zhai \cite[Theorem~7]{NingZhai2023Triangles} proved that
\[
 \lambda_1(G)\ge\lambda_1(T_2(n))
 \quad\Longrightarrow\quad
 k_3(G)\ge\left\lfloor\frac n2\right\rfloor-1,
\]
unless $G=T_2(n)$. Let $Y_{n,2,p}$ be obtained from $T_2(n)$ by adding a matching of $p$ edges inside a largest part. Li, Lu, and Peng~\cite[Theorem~2.2]{LiLuPeng2024SpectralER} established the sharp spectral Erd\H{o}s--Rademacher theorem: every $n$-vertex graph $G$ satisfying $\lambda_1(G)\ge\lambda_1(Y_{n,2,1})$ contains at least $\lfloor n/2\rfloor$ triangles. Li, Feng, and Peng~\cite{LiFengPeng2024SpectralLS} subsequently proved that, for $1\le p\le\sqrt n/11$,
the condition $\lambda_1(G)\ge\lambda_1(Y_{n,2,p})$ guarantees at least $p\lfloor n/2\rfloor$ triangles. The broader problem of spectral bounds of subgraph counts for color-critical graphs was also highlighted in the survey of Liu and Ning~\cite[Section~9]{LiuNing2023Unsolved}.

Fang, Li, Lin, and Ma~\cite{FangLiLinMa2025ColorCritical} proved spectral supersaturation theorems for all color-critical graphs, using the minimum and maximum spectral radii among graphs obtained by adding $q$ edges to $T_r(n)$ as thresholds.
More recently, Wu and Lu~\cite{WuLu2026Strict} determined the largest possible spectral radius when the number of cliques is required to be below, or at most, a prescribed bound, for $r\ge3$ and cases just above the Tur\'an threshold. They also obtained sharp thresholds using Tur\'an graphs with added pairwise disjoint edges. Chen and Li~\cite{ChenLi2026Optimal} determined asymptotically optimal ranges for the number of added edges in such comparisons, for cliques, including triangles, and for odd cycles. Together, these results give sharp lower bounds on subgraph counts under conditions on the largest adjacency eigenvalue. Besides counting triangles, one may also ask how many edges lie in triangles and how the triangles overlap. Li, Feng, and Peng studied the number of edges lying in triangles~\cite{LiFengPeng2025EFR}, as well as triangle configurations and bowties~\cite{LiFengPeng2025Bowties}, where a bowtie consists of two triangles sharing exactly one vertex.

The results above express spectral thresholds in terms of the number of vertices. A complementary approach uses the number of edges instead. This distinction is natural: adding isolated vertices changes $n$, but leaves $m$, the spectral radius, and the numbers of triangles and larger cliques unchanged. Thresholds expressed in terms of $m$ are therefore unaffected by such additions.
Li, Liu, and Zhang~\cite{li2025edgespectralerdhosstonesimonovitstheoremstability} established general bounds on the spectral radius of graphs avoiding a fixed non-bipartite graph, and described the structure of graphs close to attaining these bounds. They also obtained exact bounds for color-critical graphs of chromatic number at least four~\cite{li2026edgespectralturantheoremscolorcritical}.

For triangles, the starting point is Nosal's theorem~\cite{nosalthesis1970}: every triangle-free graph with $m$ edges satisfies $\lambda_1(G)\le\sqrt{m}$. The corresponding supersaturation question asks how many triangles must appear once this threshold is crossed. Ning and  Zhai~\cite{NingZhai2023Triangles} obtained a lower bound on this number, while Zhang and Zhai~\cite{zhang2026spectralthresholdtrianglecounting} and Chen, Li, and Tang~\cite{ChenLiTang2026Nosal} obtained stronger estimates under stronger spectral assumptions. Li, Liu, and Zhang~\cite{LiLiuZhang2026Nosal} showed that crossing Nosal's threshold forces many triangles sharing a common edge, and obtained sharp estimates for the number of $4$-cycles, together with extensions to larger cliques. Further bounds on the largest number of triangles sharing a common edge, and on the number of edges lying in triangles, appear in~\cite{ZhaiLiLou2026,zhao2026problemsbooksizetriangularedges}.

Beyond triangles, Chen and Li~\cite{ChenLi2026EdgeSpectral} proved an edge-based version of Mubayi's counting theorem for color-critical graphs of chromatic number at least four, giving asymptotically sharp lower bounds on the number of copies when the spectral radius exceeds the
corresponding threshold within a specified range. Further counting results cover complete bipartite graphs and even cycles~\cite{li2026spectralsidorenkoinequalitiesedgespectral}, as well as odd cycles of length at least five~\cite{liu2026sharpedgespectralsupersaturationodd}. Together, these works study both how many copies must appear and how they share vertices or edges. 

\subsection{Supersaturation for positive square energy: new results.}
The main aim of this paper is to begin the study of supersaturation for positive square energy, which naturally extends the idea of spectral supersaturation. The square energies were studied by Wocjan and Elphick~\cite{WocjanElphick2013}, in connection with a spectral lower bound for the chromatic number involving the entire adjacency spectrum. Elphick, Farber, Goldberg, and Wocjan~\cite{ElphickFarberGoldbergWocjan2016} subsequently initiated the systematic study of square energies as graph invariants. Elphick and Wocjan~\cite[Conjecture~1 and
Section~3.1]{ElphickWocjan2018} conjectured the positive square-energy strengthening of Wilf's bound and proved the triangle-free case. The conjecture and this case
also appear in the later work of Elphick, Linz, and Wocjan~\cite{ElphickLinzWocjan2024}. Liu, Tang, and Zhang \cite{LiuTangZhang2026} subsequently proved the
positive square-energy Tur\'{a}n theorem: every $K_{r+1}$-free graph satisfies
\[
    \sqrt{s^+}\leq\left(1-\frac{1}{r}\right) n.
\]
The equality case was studied by Jayarajan, Kannan, Pragada, and Roy \cite{JayarajanKannanPragadaRoy2026}. Our main theorem is a natural extension of the literature of spectral supersaturation to positive square energy. Precisely,
we show the following,

\begin{theorem}\label{thm:main_square_energy}
For every integer $r\geq2$ and every $\varepsilon>0$, there exists $n_0=n_0(r,\varepsilon)$ such that the following holds. Let $G$ have order $n\geq n_0$. If
\begin{equation*}
    \sqrt{s^+(G)}>\left(1-\frac1r\right)n,
\end{equation*}
then
\begin{equation*}
    k_{r+1}(G)\geq
    \left(\frac1{r^{r-1}}-\varepsilon\right)n^{r-1}.
\end{equation*}
The coefficient $r^{-(r-1)}$ is asymptotically best possible for every fixed $r\geq2$.
\end{theorem}

The above result extends asymptotically sharp spectral supersaturation from the largest adjacency eigenvalue to the combined contribution of all positive adjacency eigenvalues. To the best of our knowledge, this is the first asymptotically sharp clique supersaturation theorem based on positive square energy. Exceeding the threshold $\sqrt{s^+(G)}=(1-1/r)n$ by any positive amount, however small, forces, asymptotically, at least as many $(r+1)$-cliques as adding one edge to a balanced Tur\'an graph. The hypothesis can hold even when the largest eigenvalue lies below $\lambda_1(T_r(n))$. For example, let $G=K_{(r-1)t}\sqcup K_t$ be the disjoint union of two cliques, where $t\ge2r$ is an integer and $n=rt$. Then
$\lambda_1(G)=(r-1)t-1<\lambda_1(T_r(n))=(r-1)t,$ whereas $s^+(G)=((r-1)t-1)^2+(t-1)^2>\left(1-\frac1r\right)^2n^2.$ Thus the square-energy hypothesis holds even though the spectral-radius hypothesis fails. Therefore earlier results requiring this spectral-radius threshold do not directly apply. Thus, the theorem preserves the optimal leading coefficient while allowing the required spectral contribution to be spread across several positive eigenvalues. Our arguments use the local harmonic inequality and random-partition method from the proof of \cite{LiuTangZhang2026}.

\subsection{Overview of the proof}\label{sec:proof_overview}
The main difficulty is that the positive square energy may exceed the threshold by an arbitrarily small amount. We must turn this small spectral excess into a sharp lower bound on the number of cliques. We argue by contradiction: for fixed $r\ge2$ and $\varepsilon>0$, suppose that arbitrarily large graphs satisfy
\[
\sqrt{s^+(G)}>\left(1-\frac1r\right)n,
\qquad
k_{r+1}(G)\le
\left(\frac{1}{r^{r-1}}-\varepsilon\right)n^{r-1}.
\]
Under these assumptions, the proof follows the scheme
\[
\begin{aligned}
&\text{Clique removal and positive square-energy stability}\\
&\qquad\Longrightarrow
  \text{a nearly balanced maximum }r\text{-cut}\\
&\qquad\Longrightarrow
  \text{control of internal and missing edges}\\
&\qquad\Longrightarrow
  \text{the required clique count}.
\end{aligned}
\] 
We describe the key ideas in more detail below.

\medskip
\noindent\textbf{Stability from near equality.}
Our first ingredient is a stability theorem for positive square energy. We show that a $K_{r+1}$-free graph whose square root of positive square energy is close to the extremal value $\left(1-\frac{1}{r}\right)n$ differs from $T_r(n)$ in only $o(n^2)$ edges. The proof examines near equality in the local harmonic inequality and random-partition argument of Liu, Tang, and Zhang. These estimates produce a partition into $r$ nearly equal blocks. The last block is independent and is joined to almost every vertex outside it. Induction on $r$ then gives the required structure.

\medskip
\noindent\textbf{Reducing to a maximum cut.}
The assumed clique bound allows the clique removal lemma to delete $o(n^2)$ edges and produce a $K_{r+1}$-free graph.
The resulting change in $\sqrt{s^+}$ is only $o(n)$, so the stability theorem applies. Consequently, a partition of $G$
maximizing the number of crossing edges has parts of size $n/r+o(n)$. Write $\ell$ for the number of internal edges and
$c$ for the number of missing crossing edges. Initially, stability gives only $\ell,c=o(n^2)$; sharper control is needed
to detect the effect of even one internal edge.

\medskip
\noindent\textbf{Counting cliques when $r\ge3$.}
Let $K$ be the complete $r$-partite graph whose parts are those of a maximum $r$-cut of $G$. The reduction argument shows that these parts are nearly equal. We compare $G$ with $K$. Its Perron vector is nearly constant, which allows a
matrix perturbation estimate to give
\[
\ell\ge1,
\qquad
c=O_r(\ell).
\]
The crucial point is that the error in this estimate is $o(1)(\ell+c)$, rather than merely $o(n^2)$. The maximum-cut condition and the assumed small clique count also force every internal degree to be $o(n)$. We then count clique candidates consisting of an internal edge and one vertex from each other part. There are $(1+o(1))\ell(n/r)^{r-1}$ such candidates, and the preceding bounds show that missing crossing edges spoil only a negligible fraction. Hence
\[
k_{r+1}(G)\ge
(1-o(1))\ell\left(\frac nr\right)^{r-1},
\]
contradicting the assumed bound.

\medskip
\noindent\textbf{Why triangles need a separate argument.}
When $r=2$, the coefficient that controls missing crossing edges in the preceding estimate vanishes. We instead write the
adjacency matrix in blocks and use the singular values of the crossing block to relate positive square energy directly to
triangles meeting both parts. Under the contradiction assumption, this estimate and the maximum-cut condition force
$\ell$, $c$, and the difference between the two part sizes to remain bounded independently of $n$. Any internal edge then
lies in at least $n/2-O_\varepsilon(1)$ triangles, giving the required contradiction.

Finally, adding one edge inside a part of $T_r(n)$ when $r$
divides $n$ shows that the leading coefficient cannot be improved.

The paper is organized as follows. Section~\ref{sec:specext_preliminaries} collects the basic
spectral estimates and graph-theoretic tools used in the proofs. Section~\ref{sec:se_stability} shows that $K_{r+1}$-free graphs whose positive square energy is close to its largest possible value are close to a balanced Tur\'an graph. Section~\ref{sec:reduction} uses the clique removal lemma to reduce the main proof to graphs close to this structure.
Sections~\ref{sec:se_higher_cliques} and~\ref{sec:se_triangles} establish the required clique counts for $r\ge3$
and $r=2$, respectively. Section~\ref{sec:se_sharpness} gives an example showing that the coefficient in the main theorem cannot be improved.

\section{Spectral and extremal preliminaries}
\label{sec:specext_preliminaries}

We use the positive square-energy Tur\'{a}n theorem of Liu, Tang, and Zhang
\cite[Theorem~1.2]{LiuTangZhang2026}.

\begin{theorem}[Liu--Tang--Zhang]\label{thm:se_turan}
For every integer $r\geq2$, every $K_{r+1}$-free graph $G$ satisfies
\[
    \sqrt{s^+(G)}\leq\left(1-\frac{1}{r}\right)n.
\]
\end{theorem}

For a real symmetric matrix $Q$, write $Q_+$ for its positive spectral part and $Q_-$ for its negative spectral part. Thus, if
\[
 Q=U\operatorname{diag}(\theta_1,\ldots,\theta_N)U^{\mathsf T},
\]
then
\[
 Q_+=U\operatorname{diag}((\theta_1)_+,\ldots,(\theta_N)_+)
 U^{\mathsf T},
 \qquad (x)_+:=\max\{x,0\}
\]
and 
\[
 Q_-=U\operatorname{diag}((\theta_1)_-,\ldots,(\theta_N)_-)
 U^{\mathsf T},
 \qquad (x)_-:=\max\{-x,0\},
\]
where $U$ is orthogonal.
We use the Frobenius inner product and norm
\[
 \ip{Q}{R}=\tr(Q^{\mathsf T}R),
 \qquad
 \norm{Q}_F^2=\ip{Q}{Q}.
\]

The positive (respectively, negative) square energy of $Q$ is the sum of the squares of its positive (respectively, negative) eigenvalues. Note that $\|Q_+\|_F^2=s^+(Q)$. A real $d\times d$ matrix $Z$ is positive semidefinite, written $Z\succeq0$, if it is symmetric and $x^T Zx\ge0$ for every $x\in\mathbb{R}^d$. For positive semidefinite matrices $Z_1$ and $Z_2$,
$ \ip{Z_1}{Z_2} =\tr(Z_1^{1/2}Z_2Z_1^{1/2}) \geq 0.$ For any real symmetric matrix $Q$ we must have $Q=Q_+-Q_-$ with $Q_+,Q_-\succeq0$ and $Q_+Q_-=0$.

We now mention a few results pertaining to the Frobenius norms of the positive and the negative parts of spectral decomposition. We include short proofs for completeness.
\begin{lemma}
\label{lem:se_matrix}
For real symmetric matrices $Q,R$ of the same order, the following hold.
\begin{align} 
 \left|\sqrt{s^+(Q)}-\sqrt{s^+(R)}\right|
 &\leq \norm{Q-R}_F,
 \label{eq:Frob_1}\\
 s^+(Q+R)&\leq
 s^+(Q)+2\ip{Q_+}{R}+\norm{R}_F^2,
 \label{eq:Frob_2}\\
 s^+(Q)&=
 \max_{Z\succeq0}\left\{2\ip{Q}{Z}-\norm{Z}_F^2\right\}.
 \label{eq:Frob_3}
\end{align}
\end{lemma}

\begin{proof}

We have
\begin{equation*}
    \norm{Q-R}_F^2 = \norm{Q_+-R_+}_F^2+\norm{Q_--R_-}_F^2+2\ip{Q_+}{R_-}+2\ip{Q_-}{R_+} \geq \norm{Q_+-R_+}_F^2.
\end{equation*}
The triangle inequality now gives
\[
 \left|\norm{Q_+}_F-\norm{R_+}_F\right|
 \leq\norm{Q_+-R_+}_F
 \leq\norm{Q-R}_F,
\]
which is \eqref{eq:Frob_1}.

\begin{align*}
    \norm{Q_+ +R}_F^2 &= \norm{Q+ Q_- +R}_F^2 =\norm{(Q+R)_+ + (Q_- -(Q+R)_-)}_F^2 \\
    &= \norm{(Q+R)_+}_F^2+\norm{Q_--(Q+R)_-}_F^2+2\ip{(Q+R)_+}{Q_--(Q+R)_-}\\
    &=\norm{(Q+R)_+}_F^2+ \norm{Q_--(Q+R)_-}_F^2+2 \ip{(Q+R)_+}{Q_-}\\
    &\geq \norm{(Q+R)_+}_F^2=s^+(Q+R).
\end{align*}

On the other hand,
$$\norm{Q_+ +R}_F^2=\norm{Q_+}_F^2+2\ip{Q_+}{R}+\norm{R}_F^2.$$
Thus,
$$s^+(Q+R) \leq s^+(Q)+2\ip{Q_+}{R}+\norm{R}_F^2.$$
This proves \eqref{eq:Frob_2}.

For every positive semidefinite matrix $Z$, we have $\ip{Q_-}{Z}\ge0$, since both $Q_-$ and $Z$
are positive semidefinite. Consequently,
\begin{align*}
2\ip{Q}{Z}-\norm{Z}_F^2
&=2\ip{Q_+}{Z}
  -2\ip{Q_-}{Z}-\norm{Z}_F^2\\
&\le 2\ip{Q_+}{Z}-\norm{Z}_F^2\\
&=\norm{Q_+}_F^2-\norm{Z-Q_+}_F^2\\
&\le\norm{Q_+}_F^2.
\end{align*}
Equality holds at $Z=Q_+$, proving \eqref{eq:Frob_3}.
\end{proof}
For adjacency matrices, the variational identity \eqref{eq:Frob_3} appears in \cite[Proof of Lemma~6.1]{zhang2024extremal}; the same argument applies to arbitrary real symmetric matrices.

We use the following standard clique removal lemma,
which follows from Szemer\'edi's regularity lemma and
the graph counting lemma; see \cite[Theorems~1.1 and~1.2]{Gowers2007}. One can find the lemma, as stated, in \cite[Theorem~1]{Mubayi2010}.
\begin{lemma}[Clique removal lemma]
\label{lem:cl_removal}
Fix $q\geq3$.  If a graph $J$ of order $N$ contains $o(N^q)$ copies of
$K_q$, then one can delete $o(N^2)$ edges from $J$ and obtain a
$K_q$-free graph.
\end{lemma}

If $G'$ is obtained from $G$ by changing $d=o(n^2)$ edges, then
\[
 \norm{A(G)-A(G')}_F=\sqrt{2d}=o(n).
\]
Consequently, \eqref{eq:Frob_1} gives
\begin{equation}
 \left|\sqrt{s^+(G)}-\sqrt{s^+(G')}\right|=o(n).
 \label{eq:se_edit_transfer}
\end{equation}

\section{Positive square-energy stability}\label{sec:se_stability}

In this section, we prove a stability result that is crucial in proving the main result. The result describes what a $K_{r+1}$-free graph looks like when its positive square energy is close to the largest possible value. It shows that such a graph must be close to the balanced Tur\'an graph: its vertices can be divided into $r$ nearly equal parts, with few edges
inside the parts and almost all possible edges between different parts present. More precisely, changing only $o(n^2)$ edges is enough to obtain $T_r(n)$.

Now we make the meaning of closeness precise. F\"uredi in \cite{Furedi2015} defined the edit distance between graphs $G_1$ and $G_2$ with the same vertex set as the cardinality of the symmetric difference of the edge sets. We call two graphs on the same vertex set $o(n^2)$-edit-close if their edge sets have symmetric difference of size $o(n^2)$. 

The stability result can be stated as follows:

\begin{lemma}
\label{prop:se_stability}
Fix $r\geq2$.  Let $G_1$ be a $K_{r+1}$-free graph of order $n$.  If
\[
 \sqrt{s^+(G_1)}=\left(1-\frac{1}{r}\right) n-o(n),
\]
then $G_1$ is $o(n^2)$-edit-close to $T_r(n)$.  More explicitly, there is a partition
\[
 V(G_1)=V_1\sqcup\cdots\sqcup V_r
\]
such that
\[
 |V_i|=\frac nr+o(n)\quad(1\leq i\leq r)
\]
and
\[
 \sum_{i=1}^r |E(G_1[V_i])|
 +\sum_{1\leq i<j\leq r}
 \bigl(|V_i||V_j|-e_{G_1}(V_i,V_j)\bigr)=o(n^2).
\]
\end{lemma}

\subsection*{The Caro--Wei  process}

We apply the following process to the $K_{r+1}$-free graph $G_1$ from Lemma~\ref{prop:se_stability}. Write $V=V(G_1)$ and $E=E(G_1)$. We recall precisely the part of Liu--Tang--Zhang~\cite[see Definition 1.5]{LiuTangZhang2026} that is needed. Choose a uniformly random ordering of $V$ and put $S_1=V$.
For $i\ge1$, while $S_i\ne\varnothing$, let $x_i$ be the
first vertex of $S_i$ in the chosen ordering, and set
\[
S_{i+1}=S_i\cap N_{G_1}(x_i),
\qquad
B_i=S_i\setminus S_{i+1}.
\]
The vertices $x_1,x_2,\ldots$ are called the pivots, and the nonempty sets $B_1,B_2,\ldots$ are called the blocks. The sets $B_i$ form a random partition $\cP$ of $V$. Its pivots $x_i$ form a clique, so every outcome has at most $r$ blocks.  For
$v\in V$ and $uv\in E$, put
\[
 p_v=\Pr(v\text{ is a pivot}),
 \qquad
 q_{uv}=\Pr(u,v\text{ lie in different blocks}).
\]
These probabilities are positive.  Indeed, $v$ is a pivot if it is first in the whole ordering, and the edge $uv$ is separated if $u$ is first in the whole ordering.  Liu, Tang, and Zhang prove the local harmonic inequality \cite[Lemma~1.6]{LiuTangZhang2026}
\begin{equation}
 \frac1{p_v}+\sum_{u\in N_{G_1}(v)}\frac1{q_{uv}}\leq n
 \qquad(v\in V).
 \label{eq:local_harmonic}
\end{equation}

A real symmetric matrix $M$ is \emph{doubly nonnegative} if $M\succeq0$ and $M$ is entrywise nonnegative.  For a partition
$\mathcal Q=\{Q_1,\ldots,Q_k\}$ of $V$, define its separated $M$-mass by
\[
 W_{\mathcal Q}(M)
 :=
 \sum_{\substack{uv\in E\\
                   u,v\text{ lie in distinct blocks of }\mathcal Q}}
 M_{uv},
\]
where the sum is over unordered edges. \\
Let $\mathbf{1}$ denote the all-ones vector of the appropriate order. We use the following estimate for doubly nonnegative matrices; see Liu, Tang, and Zhang~\cite[proof of Lemma~1.4, equation~(3.2)]
{LiuTangZhang2026}.
\begin{lemma}
\label{lem:dnn_partition}
Let $M$ be a doubly nonnegative matrix indexed by $V$, let $T=\one^{\mathsf T}M\one$, and let $\mathcal{Q}=\{Q_1,\ldots,Q_k\}$ be any partition.  Then
\[
 W_{\mathcal{Q}}(M)\leq\frac12\left(1-\frac1k\right)T.
\]
\end{lemma}

\subsection*{Proof of Lemma \ref{prop:se_stability}}
For matrices $Z_1$ and $Z_2$ of the same order, let $Z_1\circ Z_2$ denote their Hadamard product. We shall prove Lemma \ref{prop:se_stability} in a series of claims. 
\begin{proof}
We proceed by induction on $r$.  
Let
\begin{equation*}
 A=A_+-A_-,
 \qquad
 A_+,A_-\succeq0,
 \qquad
 A_+A_-=A_-A_+=0
\end{equation*}
be the positive--negative spectral decomposition of $A=A(G_1)$ and define
\[
X=A_+, \qquad 
 M=X\circ X,
 \qquad
 T=\norm{X}_F^2=s^+(G_1),
 \qquad
 S=\sum_{uv\in E}|X_{uv}|.
\]
The Schur product theorem gives $M\succeq0$, and $M$ is entrywise nonnegative.  Moreover,
\begin{equation}
 \one^{\mathsf T}M\one
 =\sum_{u,v}X_{uv}^2
 =\norm{X}_F^2=T.
 \label{eq:M_total}
\end{equation}
Note that, $AX=X^2$.  Therefore
\begin{equation}
 T=\tr(AX)=2\sum_{uv \in E}X_{uv}\leq2S.
 \label{eq:T_less_2S}
\end{equation}

For the random Caro--Wei partition, put
\[
 B=\sum_v\frac1{p_v},
 \qquad
 H=\sum_{uv\in E}\frac1{q_{uv}},
 \qquad
 B'=\sum_v p_v,
\]
and
\[
 \overline W
 :=\mathbb E W_{\cP}(M)
 =\sum_{uv\in E}q_{uv}X_{uv}^2.
\]

\begin{claim}\label{cl:se_S_near}
    \begin{equation*}
 S=\frac T2+o(n^2).
\end{equation*}
\end{claim}
\begin{proof}
Summing \eqref{eq:local_harmonic} over the vertices gives
\begin{equation}
 B+2H\leq n^2.
 \label{eq:se_BH}
\end{equation}
The number of pivots is at most $r$, hence
\begin{equation}
 B'\leq r.
 \label{eq:B_prime_upper}
\end{equation}
Cauchy--Schwarz applied to $(\sqrt{p_v})_v$ and
$(1/\sqrt{p_v})_v$ gives
\begin{equation}
 BB'\geq n^2.
 \label{eq:BB_prime}
\end{equation}
It follows from \eqref{eq:se_BH}--\eqref{eq:BB_prime} that
\begin{equation}
 H\leq\frac{\left(1-\frac{1}{r}\right) n^2}{2}.
 \label{eq:H_upper}
\end{equation}

Every outcome of $\cP$ has at most $r$ blocks.  Thus
Lemma \ref{lem:dnn_partition} and \eqref{eq:M_total} imply
\begin{equation}
 \overline W\leq\frac{\left(1-\frac{1}{r}\right)T}{2}.
 \label{eq:W_upper}
\end{equation}
Cauchy--Schwarz gives
\begin{equation}
 S^2
 \leq
 \left(\sum_{uv\in E}q_{uv}X_{uv}^2\right)
 \left(\sum_{uv\in E}\frac1{q_{uv}}\right)
 =\overline WH.
 \label{eq:weighted_CS}
\end{equation}
Combining \eqref{eq:T_less_2S}, \eqref{eq:H_upper},
\eqref{eq:W_upper}, and \eqref{eq:weighted_CS}, we obtain
\begin{equation}
 \frac{T^2}{4}
 \leq S^2
 \leq\overline WH
 \leq\frac{\left(1-\frac{1}{r}\right)^2n^2T}{4}.
 \label{eq:full_chain}
\end{equation}

The hypothesis says 
\begin{equation}\label{eq:T_close}
   \sqrt{T}=\left(1-\frac{1}{r}\right)n-o(n). 
\end{equation} 
Since $r\ge2$ is fixed, \eqref{eq:T_close} gives
\[
\sqrt{T}=\left(1-\frac1r+o(1)\right)n=\Theta(n),
\qquad
T=\left(\left(1-\frac1r\right)^2+o(1)\right)n^2
=\Theta(n^2).
\]
In particular, $\sqrt{T}=O(n)$ and $T=\Omega(n^2)$.
Taking square roots in \eqref{eq:full_chain} and then using \eqref{eq:T_close}, we obtain
\begin{align*}
\frac{T}{2}
&\le S
\le \frac{\left(1-\frac1r\right)n\sqrt{T}}{2}\\
&=\frac{(\sqrt{T}+o(n))\sqrt{T}}{2}\\
&=\frac{T}{2}+\frac{\sqrt{T}\,o(n)}{2}\\
&=\frac{T}{2}+o(n^2),
\end{align*}
where the last equality uses $O(n)o(n)=o(n^2)$.
Therefore,
\begin{equation*}
S=\frac{T}{2}+o(n^2).  \qedhere
\end{equation*}
\end{proof}
\begin{claim}\label{cl:WH_close}
\begin{equation*}
 \overline W=\frac{\left(1-\frac{1}{r}\right)T}{2}+o(n^2),  \quad H=\frac{\left(1-\frac{1}{r}\right)n^2}{2}+o(n^2).
\end{equation*}
\end{claim}
\begin{proof}
By Claim~\ref{cl:se_S_near} and $T=\Theta(n^2)$, we have
\begin{equation*}
S^2=\frac{T^2}{4}+o(n^4).
\end{equation*}
Using \eqref{eq:weighted_CS}, \eqref{eq:H_upper}, and \eqref{eq:T_close}, we obtain
\[
\overline{W}
\ge \frac{S^2}{H}
\ge \frac{2S^2}{(1-1/r)n^2}
\ge \frac{(1-1/r)T}{2}-o(n^2).
\]
Together with \eqref{eq:W_upper}, this gives
$$\overline{W}=\frac{(1-1/r)T}{2}+o(n^2).$$

Similarly, \eqref{eq:weighted_CS} and \eqref{eq:W_upper} give
$$H\ge \frac{S^2}{\overline{W}} \ge \frac{2S^2}{(1-1/r)T} \ge \frac{(1-1/r)n^2}{2}-o(n^2),$$
where division of an $o(n^4)$ term by $T=\Theta(n^2)$
produces an $o(n^2)$ term.
Combining this with \eqref{eq:H_upper}, we obtain
\begin{equation*}
H=\frac{(1-1/r)n^2}{2}+o(n^2). \qedhere
\end{equation*}
\end{proof}
Claim~\ref{cl:WH_close}, together with \eqref{eq:se_BH}, \eqref{eq:B_prime_upper}, and
 \eqref{eq:BB_prime}, now implies
\begin{equation}
 B=\frac{n^2}{r}+o(n^2),
 \qquad
 B'=r-o(1).
 \label{eq:BB_prime_close}
\end{equation}
 
For $v\in V(G_1)$, define $ \mu_v=(M\one)_v$.
\begin{claim}\label{cl:mu_close}
    For the graph $G_1$, we must have $\sum \limits_v|\mu_v-\left(1-\frac{1}{r}\right)^2n|=o(n^2).$
\end{claim}
\begin{proof}
Set 
\[
 \tau_v=\frac1{p_v},
 \qquad a=\frac nr.
\]
By \eqref{eq:BB_prime_close}, we get
\begin{align*}
 \sum_v\frac{(\tau_v-a)^2}{\tau_v}
 &=\sum_v\left(\tau_v-2a+\frac{a^2}{\tau_v}\right)\\
 &=B-2an+a^2B'=\frac{n^2}{r}+o(n^2)-2\frac{n^2}{r}+\frac{n^2}{r}-\frac{n^2}{r^2}o(1)
 =o(n^2).
\end{align*}
Cauchy--Schwarz and the bound $B \leq n^2$ give
\begin{equation*}
 \sum_v|\tau_v-a|
 \leq
 \left(\sum_v\frac{(\tau_v-a)^2}{\tau_v}\right)^{1/2}
 \left(\sum_v\tau_v\right)^{1/2}
 =o(n^2).
\end{equation*}
Let 
\[
 L_v=\sum_{u\in N_{G_1}(v)}\frac1{q_{uv}},
 \qquad
 \sigma_v=n-\tau_v-L_v.
\]
By \eqref{eq:local_harmonic}, $\sigma_v\geq0$, and by
\eqref{eq:se_BH}, Claim~\ref{cl:WH_close}, and \eqref{eq:BB_prime_close},
\[
 \sum_v\sigma_v=n^2-B-2H=o(n^2).
\]
Since $L_v-\left(1-\frac{1}{r}\right)n=a-\tau_v-\sigma_v$, we obtain
\begin{equation}
 \sum_v|L_v-\left(1-\frac{1}{r}\right)n|\leq \sum_v (|\tau_v-a|+\sigma_v)=o(n^2).
 \label{eq:L_close}
\end{equation}

Put
\[
\kappa=\frac{S}{H}.
\]
By \eqref{eq:T_close}, Claim~\ref{cl:se_S_near} and Claim~\ref{cl:WH_close}, 
$\kappa=\left(1-\frac{1}{r}\right)+o(1).$
Expanding the square we obtain,
\begin{align}
 \sum_{uv\in E}q_{uv}
 \left(|X_{uv}|-\frac{\kappa}{q_{uv}}\right)^2
 &=\sum_{uv\in E}q_{uv}
 \left(|X_{uv}|^2+\frac{\kappa^2}{q_{uv}^2}-2|X_{uv}|\frac{\kappa}{q_{uv}}\right) \notag \\
 &=\overline W+\kappa^2H-2\kappa S=\overline W-\frac{S^2}{H} \notag \\
 &=\overline W-\kappa S    \notag \\
 &= \frac{\left(1-\frac{1}{r}\right)T}{2}+o(n^2) -\left(\left(1-\frac{1}{r}\right)+o(1)\right)\bigg(\frac{T}{2}+o(n^2)\bigg) \notag \\
 &=o(n^2)
 \label{eq:qv_close}
\end{align}
where the second last equality follows from Claim~\ref{cl:se_S_near} and Claim~\ref{cl:WH_close}.
Applying Cauchy--Schwarz and using \eqref{eq:qv_close} and 
$H=O(n^2)$ gives
\begin{align}
 \sum_{uv\in E}
 \left||X_{uv}|-\frac{\kappa}{q_{uv}}\right| &\leq \sqrt{\sum_{uv\in E} q_{uv}\left  (\left||X_{uv}|-\frac{\kappa}{q_{uv}}\right|^2 \right)}\sqrt{\sum_{uv\in E}\frac{1}{q_{uv}}}  \notag \\
 &=o(n^2).
 \label{eq:X_uv_close}
\end{align}
Also, by \eqref{eq:T_less_2S} and Claim~\ref{cl:se_S_near},
\begin{equation}
 \sum_{uv\in E}(|X_{uv}|-X_{uv})
 =S-\frac T2=o(n^2).
 \label{eq:X_uv_diff}
\end{equation}

 Since $AX=X^2$,
\begin{equation*}
 \mu_v=\sum_uX_{vu}^2=(X^2)_{vv}=(AX)_{vv}
 =\sum_{u\in N_{G_1}(v)}X_{uv}.
\end{equation*}
Equations \eqref{eq:X_uv_close} and \eqref{eq:X_uv_diff}, counted at
the two endpoints of every edge, give
\begin{align*}
\sum_v|\mu_v-\kappa L_v|
&=\sum_v
  \left|
  \sum_{u\in N_{G_1}(v)}
  \left(X_{uv}-\frac{\kappa}{q_{uv}}\right)
  \right|\\
&\le \sum_v\sum_{u\in N_{G_1}(v)}
  \left|X_{uv}-\frac{\kappa}{q_{uv}}\right|\\
&\le \sum_v\sum_{u\in N_{G_1}(v)}
  \left(
  |X_{uv}|-X_{uv}
  +\left||X_{uv}|-\frac{\kappa}{q_{uv}}\right|
  \right)\\
&=2\sum_{\{u,v\}\in E}
  \bigl(|X_{uv}|-X_{uv}\bigr)\\
&\quad
  +2\sum_{uv \in E}
  \left||X_{uv}|-\frac{\kappa}{q_{uv}}\right|\\
&=o(n^2).
\end{align*}

Furthermore, \eqref{eq:L_close} and $\kappa=\left(1-\frac{1}{r}\right)+o(1)$ give
\begin{align*}
 \sum_v \left|\kappa L_v-\left(1-\frac{1}{r}\right)^2n\right|&=\sum_v \left|\kappa\left(L_v-\left(1-\frac{1}{r}\right) n\right)+\left(1-\frac{1}{r}\right)n\left(\kappa-\left(1-\frac{1}{r}\right)\right) \right|  \\
 &\leq |\kappa|\sum_v \left|L_v-\left(1-\frac{1}{r}\right)n \right|
      +\left(1-\frac{1}{r}\right) n^2\left|\kappa-\left(1-\frac{1}{r}\right)\right|\\
 &=o(n^2).
\end{align*}
Consequently,
\begin{equation*}
 \sum_v \left |\mu_v-\left(1-\frac{1}{r}\right)^2n \right |=o(n^2). \qedhere
\end{equation*}
\end{proof}
For an outcome $\cP$ of the Caro--Wei process, define its deficit by
\[
 \Delta(\cP)=\frac{\left(1-\frac{1}{r}\right)T}{2}-W_{\cP}(M)\geq0.
\]
We next show that the process can divide the vertices into exactly $r$ blocks while keeping the deficit negligible compared with $n^2$.
\begin{claim}\label{cl:delt_close}
    There is an outcome $\cP=\{B_1, \dots, B_r\}$ of the Caro--Wei process with $r$ blocks such that $\Delta(\cP)=o(n^2).$ 
\end{claim}
\begin{proof}
By Claim~\ref{cl:WH_close},
\[
 \mathbb E\Delta(\cP)=\frac{\left(1-\frac{1}{r}\right)T}{2} - \mathbb EW_{\cP}(M)=o(n^2).
\]
If $\cP$ has at most $r-1$ blocks, then
Lemma \ref{lem:dnn_partition} gives
\[
 \Delta(\cP)
 \geq\frac{\left(1-\frac{1}{r}\right)-\left(1-\frac{1}{r-1}\right)}2T.
\]
  
From \eqref{eq:T_close} we get, $$\Delta(\cP)=\Omega_{r}(n^2)$$
when $r\geq3$.
For $r=2$, the same conclusion follows because a
one-block outcome has $W_{\cP}(M)=0$. 
Put
\[
 \rho_n=\frac{\mathbb E\Delta(\cP)}{n^2}=o(1).
\]
If $\rho_n=0$, then $\Delta(\cP)=0$ almost surely, so every outcome of
positive probability has $r$ blocks and we may choose any such outcome.
If $\rho_n>0$, Markov's inequality gives
\[
 \Pr\bigl(\Delta(\cP)>\sqrt{\rho_n}\,n^2\bigr) \leq \frac{\mathbb E\Delta(\cP)}{\sqrt{\rho_n}\,n^2}=\frac{\rho_n\,n^2}{\sqrt{\rho_n}\,n^2}
 =\sqrt{\rho_n}=o(1),
\]
whereas the preceding $\Omega_r(n^2)$ gap argument gives
$\Pr(\cP\text{ has fewer than }r\text{ blocks})=o(1)$.  Hence by union bound, an
$r$-block outcome with
\begin{equation*}
 \Delta(\cP)=o(n^2)
\end{equation*}
exists.
\end{proof}
We shall work with the outcome obtained from Claim~\ref{cl:delt_close}. For $V' \subseteq V(G_1)$ let $\one_{V'} \in \mathbb{R}^{V(G_1)}$ denote its indicator vector. Factor $M=R^{\mathsf T}R$ and put
\[
 g_i=R\one_{B_i},
 \qquad h=R\one.
\]
Define
\[
 N_{\cP}(M)=
 \sum_{\substack{u<v,\ uv\notin E\\
                  u,v\text{ in different blocks}}}M_{uv},
\]
which adds up the entries above the diagonal in $M$ corresponding to non-edges in different blocks.
\begin{claim}
    With the above notation, the following holds for $1 \leq i \leq r$,
\begin{align}
    \norm{g_i -\frac hr}&=o(n),  \label{eq:g_i-hr} \\
    \sum \limits_{v \in B_i}\mu_v&=\frac Tr +o(n^2),  \label{eq:block_mu} \\
    |B_i| &=\frac nr +o(n).  \label{eq:block_size} 
\end{align}

\end{claim}
\begin{proof}
Every diagonal entry of $M$ belongs to a block, while each unordered pair of vertices in different blocks is counted twice in $\mathbf{1}^{T}M\mathbf{1}-\sum \limits_{i=1}^r\mathbf{1}_{B_i}^{T}M\mathbf{1}_{B_i}.$

Separating these pairs into edges and non-edges gives
\begin{align*}
W_{\mathcal P}(M)+N_{\mathcal P}(M)
&=\frac12\left(
\mathbf{1}^{T}M\mathbf{1}
-\sum_{i=1}^r\mathbf{1}_{B_i}^{T}M\mathbf{1}_{B_i}
\right)\\
&=\frac12\left(T-\sum_{i=1}^r\|g_i\|^2\right).
\end{align*}
Consequently,
\begin{equation*}
W_{\mathcal P}(M)
=\frac12\left(T-\sum_{i=1}^r\|g_i\|^2\right)
-N_{\mathcal P}(M).
\end{equation*}
Therefore, 
\begin{align}
 \Delta(\cP)
 &=\frac{\left(1-\frac{1}{r}\right)T}{2}-W_{\cP}(M)   \notag \\
 &=N_{\cP}(M)+\frac 12\sum_{i=1}^r\norm{g_i}^2-\frac{T}{2r} \notag \\
 &=N_{\cP}(M)+\frac1{2r}\left(r
                 \sum_{i=1}^r\norm{g_i}^2-\norm{\sum_{i=1}^r g_i}^2\right)  \notag \\
 &=N_{\cP}(M)+\frac1{2r}\sum_{1\leq i<j\leq r}\norm{g_i-g_j}^2.
 \label{eq:delt_P_gdiff}
\end{align}

It follows from Cauchy--Schwarz, Claim~\ref{cl:delt_close} and
\eqref{eq:delt_P_gdiff} that for $1\leq i\leq r$,
\begin{align*}
 \norm{g_i-h/r}^2&=\norm{\frac{1}{r}\sum \limits_{j=1}^r (g_i-g_j)}^2 \\
 &\leq \frac1r \sum \limits_{j=1}^r \norm{g_i-g_j}^2  \\
 &\leq \frac1r \sum_{1 \leq a < b \leq r}\norm{g_a-g_b}^2\\
 &=2(\Delta(\cP)-N_{\cP}(M))  \notag \\
 &\leq 2\Delta(\cP) \\
 &=o(n^2)
\end{align*}
Thus for $1\leq i\leq r$, $\norm{g_i-h/r}=o(n)$.
Here $\norm h=\sqrt T=O(n)$.  Therefore
\begin{align*}
\left|\sum_{v\in B_i}\mu_v-\frac{T}{r}\right|
&=\left|\one_{B_i}^{\mathsf T}M\one-\frac{T}{r}\right|\\
&=\left|\one_{B_i}^{\mathsf T}R^{\mathsf T}R\one
        -\frac{\norm{h}^2}{r}\right|\\
&=\left|\ip{R\one_{B_i}}{R\one}
        -\frac{\ip{h}{h}}{r}\right|\\
&=\left|\ip{g_i}{h}-\ip{\frac{h}{r}}{h}\right|\\
&=\left|\ip{g_i-\frac{h}{r}}{h}\right|\\
&\le \norm{g_i-\frac{h}{r}}\,\norm{h}
\qquad\text{(Cauchy--Schwarz)}\\
&=o(n^2).
\end{align*}
Consequently,
\[
\sum_{v\in B_i}\mu_v=\frac{T}{r}+o(n^2).
\]
By Claim~\ref{cl:mu_close} we get,
\begin{align*}
\left |\sum_{v\in B_i} \mu_v-\left(1-\frac{1}{r}\right)^2n |B_i| \right |&\leq \sum_{v \in B_i}\left | \mu_v -\left(1-\frac{1}{r}\right)^2 n \right| \\
&\leq \sum_{v \in V(G_1)}\left | \mu_v -\left(1-\frac{1}{r}\right)^2 n \right|=o(n^2)
\end{align*}
Hence 
$$\sum_{v \in B_i} \mu_v=\left(1-\frac{1}{r}\right)^2n |B_i|+o(n^2)$$
Combining this with \eqref{eq:block_mu}, we obtain
$$\left(1-\frac{1}{r}\right)^2n |B_i|=\frac Tr +o(n^2)$$
Using \eqref{eq:T_close} we get,
\begin{equation*}
 |B_i|=\frac nr+o(n)
 \qquad(1\leq i\leq r). \qedhere
\end{equation*}
\end{proof}
We next show that the last block has no edges within it and that almost all possible edges between this block and the remaining vertices are present.
\begin{claim}\label{cl:h_0}
    Let $C$ be the last block of the outcome obtained from Claim ~\ref{cl:delt_close} and put $U=V(G_1)\setminus C$. Then $C$ is independent and $\sum \limits_{\substack{wu\in E\\w\in C,\ u\in U}}X_{wu}=\frac Tr+o(n^2)$. Moreover, if $ h_0:=|C||U|-e_{G_1}(C,U)$ then $h_0=o(n^2)$.
\end{claim}
\begin{proof}
The first
$r-1$ pivots form a $K_{r-1}$, and $C$ is their common neighborhood.
Thus $C$ is independent since an edge inside $C$, together with those
$r-1$ pivots, would form a $K_{r+1}$.

Because $C$ is independent, the identity $\mu_v=\sum \limits_{u\in N_{G_1}(v)}X_{uv}$ and
\eqref{eq:block_mu} give
\begin{equation}
 \sum_{\substack{wu\in E\\w\in C,\ u\in U}}X_{wu}
 =\sum_{w\in C}\mu_w
 =\frac Tr+o(n^2).
 \label{eq:cut_X_sum}
\end{equation}
Moreover using \eqref{eq:g_i-hr},
\begin{align}
 \sum_{w\in C,\ u\in U}X_{wu}^2
 &=\one_C^{\mathsf T}M\one_U
 =\ip{g_r}{h-g_r}\notag\\
 &=\ip{\frac hr}{\left(1-\frac{1}{r}\right) h}+o(n^2)=\frac{\left(1-\frac{1}{r}\right)T}{r}+o(n^2).
 \label{eq:cut_X_square}
\end{align}
Cauchy--Schwarz, followed by \eqref{eq:cut_X_sum} and
\eqref{eq:cut_X_square}, yields
\begin{align*}
 e_{G_1}(C,U)
 &\geq
 \frac{\left(\sum \limits_{wu\in E_{G_1}(C,U)}X_{wu}\right)^2}
      {\sum \limits_{w\in C,\ u\in U}X_{wu}^2}    \\
 &= \frac{\frac{T^2}{r^2}+o(n^4)}{\frac{\left(1-\frac{1}{r}\right)T}{r}+o(n^2)}=T \frac{\frac{1}{r^2}+o(1)}{\frac{\left(1-\frac{1}{r}\right)}{r}+o(1)}\\
 &=T \left(\frac{1}{\left(1-\frac{1}{r}\right) r}+o(1) \right)\\
 &=\frac{T}{\left(1-\frac{1}{r}\right)r}+o(n^2)
 =\frac{\left(1-\frac{1}{r}\right)}{r}n^2-o(n^2).
\end{align*}
The denominator in the first inequality is assured to be positive by $\eqref{eq:T_close}$ and $\eqref{eq:cut_X_square}$. On the other hand, \eqref{eq:block_size} gives
\[
 |C||U|=\frac{\left(1-\frac{1}{r}\right)}{r}n^2+o(n^2).
\]
Hence the number of missing edges $h_0$ between $C$ and $U$ is $o(n^2)$.
\end{proof}
\begin{claim}\label{cl:U_n_r_cliques}
    The graph induced by $U$ has $o(n^r)$ copies of $K_r$.
\end{claim}
\begin{proof}
We double-count the missing edges
from $C$ to a clique $J$. For every $w\in C$ and every $r$-clique $J$ in $G_1[U]$, at least one edge between $w$ and $V(J)$ is missing; otherwise $V(J)\cup\{w\}$ would induce a copy of
$K_{r+1}$. There are at least $|C|k_r(G_1[U])$ triples $(w,J,u)$, where $w\in C$, $J$ is an $r$-clique in $G_1[U]$, $u\in V(J)$, and $wu\notin E$. On the other hand, each missing edge $wu$ can be associated with at most $\binom{|U|-1}{r-1}$ such $r$-cliques $J$ containing $u$. Therefore,
\begin{equation*}
 |C|\,k_r(G_1[U])
 \leq h_0\binom{|U|-1}{r-1}=o(n^{r+1}).
\end{equation*}
Here $h_0=o(n^2)$ by Claim~\ref{cl:h_0}, and
$\binom{|U|-1}{r-1}=O(n^{r-1})$. Next since $|C|=\Theta(n)$, we have
\begin{equation*}
 k_r(G_1[U])=o(n^r). \qedhere
\end{equation*}
\end{proof}
If $r=2$, Claim~\ref{cl:U_n_r_cliques} gives $|E(G_1[U])|=o(n^2)$.
By Claim~\ref{cl:h_0}, $C$ is independent and only $o(n^2)$ edges
are missing between $C$ and $U$.
Together with $|C|,|U|=n/2+o(n)$ from \eqref{eq:block_size},
this proves the lemma for $r=2$.

For $r\ge3$, having few copies of $K_r$ does not imply having few edges, so a different argument is needed. We next show that deleting $o(n^2)$ edges removes all copies of $K_r$ while keeping the positive square energy close to its largest possible value, allowing us to apply the induction hypothesis to the remaining graph.
\begin{claim}\label{cl:ind_ready}
    For $r \geq 3$, there exists a $K_r$-free graph $F$ obtained from $G_1[U]$ by deleting $o(n^2)$
    edges such that $\sqrt{s^+(F)}=\left(1-\frac{1}{r-1}\right)|U|-o(|U|).$
\end{claim}
\begin{proof}
Assume that $r\geq3$. By Lemma \ref{lem:cl_removal} and
Claim~\ref{cl:U_n_r_cliques}, deleting $o(n^2)$ edges from $G_1[U]$ produces a
$K_r$-free graph $F$.

  Let
\[
 A_U=A(G_1[U])
\]
and $Y$ be the principal submatrix of $X$ obtained by retaining rows and columns indexed by $U$. Since $X$ is a positive semidefinite matrix, the principal submatrix $Y$ is also positive semidefinite.  From
$\norm{g_r-h/r}=o(n)$ we obtain
\begin{equation}
 \norm{Y}_F^2=\sum_{u,v \in U}X_{uv}^2
 =\one_U^{\mathsf T}M\one_U=\norm{R\one_U}^2
 =\norm{h-g_r}^2
 =\left(1-\frac{1}{r}\right)^2T+o(n^2).
 \label{eq:Y_norm}
\end{equation}
Recall by \eqref{eq:T_less_2S} that $T=\ip{A}{X}$. By Claim~\ref{cl:h_0}, $C$ is independent. Therefore $T=\ip{A_U}{Y}+2\sum \limits_{wu\in E_{G_1}(C,U)}X_{wu}$.  \eqref{eq:cut_X_sum} gives
\begin{equation}
 \ip{A_U}{Y}
 =T-2\sum \limits_{wu\in E_{G_1}(C,U)}X_{wu}
 =\frac{r-2}{r}T+o(n^2).
 \label{eq:AUY}
\end{equation}
By \eqref{eq:Y_norm}, for sufficiently large $n$, $\norm{Y}_F^2 > 0$. Since $r \geq 3$, \eqref{eq:AUY} and $T=\Theta(n^2)$ imply $\ip{A_U}{Y}>0$ for sufficiently large $n$. To use \eqref{eq:Frob_3} with $tY$, where $t\geq0$ is optimized, note that the expression maximized in  \eqref{eq:Frob_3} with $Z=tY$ equals $2\ip{A_U}{Y}t-
\norm{Y}_F^2t^2=\frac{\ip{A_U}{Y}^2}{\norm{Y}_F^2}-\norm{Y}_F^2\left( t- \frac {\ip{A_U}{Y}}{\norm{Y}_F^2} \right)^2$ which is maximized at $t=\frac{\ip{A_U}{Y}}{\norm{Y}_F^2}>0$. Therefore, using \eqref{eq:Frob_3} with $\frac{\ip{A_U}{Y}}{\norm{Y}_F^2}Y$ and using \eqref{eq:Y_norm}--\eqref{eq:AUY}, we find
\begin{align*}
 s^+(G_1[U])
 &\geq\frac{\ip{A_U}{Y}^2}{\norm{Y}_F^2}\notag\\
 &=\frac{\left(\frac{r-2}{r}T+o(n^2)\right)^2}{\left(1-\frac{1}{r}\right)^2T+o(n^2)} \notag\\
 &=\frac{\left(\frac{r-2}{r}\right)^2T^2+o(n^4)}{\left(1-\frac{1}{r}\right)^2T+o(n^2)} \notag \\
 &=T \frac{\left(\frac{r-2}{r}\right)^2+o(1)}{\left(1-\frac{1}{r}\right)^2+o(1)}  \notag \\
 &=\left(\frac{r-2}{r-1}\right)^2T-o(n^2)\notag\\
 &=\left(1-\frac{1}{r-1}\right)^2|U|^2-o(n^2).
\end{align*}
Here $|U|=\left(1-\frac{1}{r}\right)n+o(n)=\Theta(n)$.  By
\eqref{eq:se_edit_transfer}, the deletion of $o(n^2)$ edges changes
$\sqrt{s^+}$ by $o(n)$.  Hence
\[
 \sqrt{s^+(F)}\geq\left(1-\frac{1}{r-1}\right)|U|-o(n).
\]
Also since $|U|=\Theta(n)$, we get, $\sqrt{s^+(F)}\geq\left(1-\frac{1}{r-1}\right)|U|-o(|U|).$
The positive square-energy Tur\'{a}n theorem applied to the $K_r$-free
graph $F$ gives
\[
 \sqrt{s^+(F)}\leq\left(1-\frac{1}{r-1}\right)|U|. \qedhere
\] \end{proof}
The lower bound in Claim~\ref{cl:ind_ready}, together with the upper bound from Theorem~\ref{thm:se_turan}, allows us to apply the induction hypothesis to $F$. Thus $F$, and hence $G_1[U]$, is $o(n^2)$-edit-close to $T_{r-1}(|U|)$. By Claim~\ref{cl:h_0}, $C$ is independent and $h_0=o(n^2)$. Together with \eqref{eq:block_size}, this shows that $G_1$ is
$o(n^2)$-edit-close to a complete $r$-partite graph whose parts all have size $n/r+o(n)$. Moving $o(n)$ vertices between its parts makes the sizes differ by at most one and changes only $o(n^2)$ pairs.  Thus $G_1$ is $o(n^2)$-edit-close to $T_r(n)$, completing the induction.
\end{proof}

\section{Reduction to a stable maximum cut}\label{sec:reduction}

We now prove Theorem \ref{thm:main_square_energy} in a series of claims.  Fix $r\geq2$.  It is enough to consider
$0<\varepsilon<r^{-(r-1)}$.  Suppose, for a contradiction, that there
is a sequence of graphs $G$ of order $n\to\infty$ such that
\begin{equation}
 \sqrt{s^+}>\left(1-\frac{1}{r}\right)n
 \label{eq:se_strict_threshold}
\end{equation}
but
\begin{equation}
 k_{r+1}
 \leq\left(\frac1{r^{r-1}}-\varepsilon\right)n^{r-1}.
 \label{eq:se_contra_count}
\end{equation}
We first show that, under these assumptions, $G$ can be turned into a Tur\'an graph $T_r(n)$ by adding and deleting a number of edges negligible compared with $n^2$.
\begin{claim}\label{cl:G_editclose}
    Under the assumptions \eqref{eq:se_strict_threshold}-\eqref{eq:se_contra_count}, $G$ is $o(n^2)$-edit-close to a copy of $T_r(n)$.
\end{claim}
\begin{proof}
By \eqref{eq:se_contra_count},
\[
 k_{r+1}=O(n^{r-1})=o(n^{r+1}).
\]
By Lemma \ref{lem:cl_removal}, one can delete $o(n^2)$ edges from $G$ and
obtain a $K_{r+1}$-free graph $G'$.  Equations
\eqref{eq:se_strict_threshold} and \eqref{eq:se_edit_transfer} give
\[
 \sqrt{s^+(G')}\geq\left(1-\frac{1}{r}\right)n-o(n).
\]
The positive square-energy Tur\'{a}n theorem gives
\[
 \sqrt{s^+(G')}\leq\left(1-\frac{1}{r}\right)n.
\]
Thus Lemma~\ref{prop:se_stability} applies to $G'$.  Since $G$ and $G'$ differ in
$o(n^2)$ edges, it follows that $G$ itself is $o(n^2)$-edit-close to
$T_r(n)$.
\end{proof}
Choose a partition
\begin{equation}
 V(G)=V_1\sqcup\cdots\sqcup V_r
 \label{eq:max_r_cut}
\end{equation}
which maximizes the number of crossing edges.  Put
\[
 n_i=|V_i|,
 \qquad
 \ell=\sum_{i=1}^r |E(G[V_i])|,
\]
and let
\[
 c=\sum_{1\leq i<j\leq r}
 \bigl(n_in_j-e_G(V_i,V_j)\bigr)
\]
be the number of missing crossing edges.
For $v\in V_i$, define
\[
N_I(v)=N_G(v)\cap V_i,
\qquad
d_I(v)=|N_I(v)|,
\]
and
\[
N_c(v)=N_G(v)\setminus V_i,
\qquad
d_c(v)=|N_c(v)|.
\]
Let
\[
d_M(v)
=\sum_{j\ne i}\bigl(|V_j|-d_G(v,V_j)\bigr)
\]
be the number of missing neighbors of $v$ in the other
parts.
\begin{claim}\label{cl:cn_il}
    For the maximum $r$-cut \eqref{eq:max_r_cut}, the following must hold:
\begin{align}
  c&=o(n^2),
 \label{eq:c_small} \\
n_i&=\frac nr+o(n)
 \qquad(1\leq i\leq r),
 \label{eq:max_part_sizes} \\
 \ell&=o(n^2).
 \label{eq:ell_small}
\end{align}   
\end{claim}
\begin{proof}
By Claim~\ref{cl:G_editclose}, there is a balanced $r$-partition with
$t_r(n)-o(n^2)$ crossing edges, say $t_r(n)-d_n$. Therefore the maximum partition
\eqref{eq:max_r_cut} has at least this many crossing edges.  Since
\[
 \sum_{i<j}n_in_j\leq t_r(n),
\]
by combining these observations, we obtain 
\begin{align*}
t_r(n)-d_n
&\le \sum_{i<j}e_G(V_i,V_j)\\
&=\sum_{i<j}n_i n_j-c\\
&\le \sum_{i<j}n_i n_j
\le t_r(n).
\end{align*}
Consequently,
\[
0\le c
=\sum_{i<j}n_i n_j-\sum_{i<j}e_G(V_i,V_j)
\le d_n=o(n^2).
\]
Moreover,
\[
 \sum_{i<j}n_in_j
 =\frac12\left(n^2-\sum_i n_i^2\right)
 =\frac{\left(1-\frac{1}{r}\right)n^2}{2}
   -\frac12\sum_i\left(n_i-\frac nr\right)^2.
\]
The preceding bounds give
\[
\sum_{i<j}n_i n_j=t_r(n)-o(n^2).
\]
Since $r$ is fixed,
\[
t_r(n)=\left(1-\frac1r\right)\frac{n^2}{2}+O_r(1),
\]
and the term $O_r(1)$ is in particular $o(n^2)$.
Using the identity above, we obtain
\begin{align*}
\frac12\sum_{i=1}^r\left(n_i-\frac nr\right)^2
&=\left(1-\frac1r\right)\frac{n^2}{2}
  -\sum_{i<j}n_i n_j\\
&=\left(1-\frac1r\right)\frac{n^2}{2}
  -t_r(n)+o(n^2)\\
&=O_r(1)+o(n^2)\\
&=o(n^2).
\end{align*}
Every summand is nonnegative, so for each $1\leq i\leq r$,
\[
0\leq\left(n_i-\frac nr\right)^2
\leq\sum_{j=1}^r\left(n_j-\frac nr\right)^2
=o(n^2).
\]
Taking square roots gives $\left|n_i-n/r\right|=o(n)$,
and hence
\[
n_i=\frac nr+o(n)
\qquad(1\leq i\leq r).
\]

Finally, Claim~\ref{cl:G_editclose} gives $m=t_r(n)+o(n^2)$, whereas the number
of crossing edges in \eqref{eq:max_r_cut} is $t_r(n)-o(n^2)$.  Thus
\begin{equation*}
 \ell=o(n^2). \qedhere
\end{equation*} \end{proof}
We now sharpen these rough estimates.  The arguments for $r\geq3$ and
$r=2$ are different.

\section{The sharp count for \texorpdfstring{$r\geq3$}{r >= 3}}\label{sec:se_higher_cliques}

Assume throughout this section that $r\geq3$.  Let $K$ be the complete $r$-partite graph with the
same parts $V_1,\ldots,V_r$ as in \eqref{eq:max_r_cut}.  Let $A_K$ be
its adjacency matrix, and put
\[
 E_{\mathrm{err}}=A(G)-A_K.
\]
Every internal edge of $G$ contributes two entries $+1$ to $E_{\mathrm{err}}$, and
every missing crossing edge contributes two entries $-1$.  Hence
\begin{equation}
 \norm{E_{\mathrm{err}}}_F^2=2(\ell+c).
 \label{eq:E_norm}
\end{equation}

By \eqref{eq:max_part_sizes}, all parts of $K$ are nonempty for sufficiently large $n$. Since $K$ is complete multipartite, it has exactly one positive eigenvalue, namely its spectral radius $\lambda_K=\lambda_1(K)$ (see \cite[Theorem 6.7]{MR572262}).
Consequently,
\[
s^+(A_K)=\lambda_K^2.
\]
If $x^{(K)}$ is a positive unit Perron vector of $K$, then
\[
(A_K)_+=\lambda_K x^{(K)}(x^{(K)})^T.
\]
\begin{claim}
    With the definitions made above, the following holds.
    \begin{equation*}
 x^{(K)}_v=\frac{1+o(1)}{\sqrt n}
 \quad\text{uniformly in }v.
\end{equation*}

\begin{equation}
 ((A_K)_+)_{uv}=\lambda_K x^{(K)}_ux^{(K)}_v=\left(1-\frac{1}{r}\right)+o(1)
 \quad\text{uniformly in }u,v.
 \label{eq:ABplus_uniform}
\end{equation}
\end{claim}
\begin{proof}
By \eqref{eq:max_part_sizes},
\[
 |E(K)|=\sum_{i<j}n_in_j=\frac{\left(1-\frac{1}{r}\right) n^2}{2}-o(n^2).
\]
Consequently,
\[
 \lambda_K\geq\frac{2|E(K)|}n=\left(1-\frac{1}{r}\right) n-o(n).
\]
Wilf's spectral Tur\'{a}n inequality \cite{Wilf1986} gives $\lambda_K\leq\left(1-\frac{1}{r}\right) n$, so
\begin{equation}
 \lambda_K=\left(1-\frac{1}{r}\right) n+o(n).
 \label{eq:rho_near}
\end{equation}
Vertices in the same part have equal Perron coordinates.  Write that
coordinate as $a_i$ on $V_i$, and put $S_K=\sum_vx^{(K)}_v$.  The Perron
equations give
\[
 \lambda_K a_i=S_K-n_ia_i,
 \qquad
 a_i=\frac{S_K}{\lambda_K+n_i}.
\]
Since $r$ is fixed, the maximum of the finitely many absolute errors
in \eqref{eq:max_part_sizes} is still $o(n)$, so \eqref{eq:rho_near} gives
$\max_{1\le i\le r}|\lambda_K+n_i-n|=o(n)$, that is,
$\lambda_K+n_i=n+o(n)$ uniformly in $i$. Thus we get,
$$a_i=\frac{S_K}{n(1+o(1))}=\frac{S_K}{n}(1+o(1)) \quad\text{uniformly in }i.$$
The normalization
$\sum_i n_ia_i^2=1$ therefore gives
\begin{equation*}
    1=\sum_i n_i\frac{S_K^2(1+o(1))}{n^2}=\frac{S_K^2(1+o(1))}{n^2}n.
\end{equation*}
This gives,
$$S_K=\sqrt{n}(1+o(1)),$$
which yields,
$$a_i=\frac{\sqrt{n}(1+o(1))}{n(1+o(1))}=\frac{1+o(1)}{\sqrt{n}} \quad\text{uniformly in }i.$$
Since every vertex in $V_i$ has coordinate $a_i$ we can write,
\begin{equation*}
 x^{(K)}_v=\frac{1+o(1)}{\sqrt n}
 \quad\text{uniformly in }v.
\end{equation*}
It follows that
\begin{equation*}
 ((A_K)_+)_{uv}=\lambda_K x^{(K)}_ux^{(K)}_v=\left(1-\frac{1}{r}\right)+o(1)
 \quad\text{uniformly in }u,v. \qedhere
\end{equation*}\end{proof}

\begin{claim}\label{cl:ell_c_bounds}
    For sufficiently large $n$,
\begin{equation*}
 \ell \geq1,
 \qquad
 c\leq
 \left(\frac{4\left(1-\frac{1}{r}\right)+2}{4\left(1-\frac{1}{r}\right)-2}+o(1)\right)\ell
 =\left(\frac{3r-2}{r-2}+o(1)\right)\ell.
\end{equation*}
In particular, $c=O_r(\ell).$
\end{claim}
\begin{proof}
Applying \eqref{eq:Frob_2} with $Q=A_K$ and $R=E_{\mathrm{err}}$ we get  
$$s^+ \leq \lambda_K^2 + 2 \ip{(A_K)_+}{E_{\mathrm{err}}}+ \norm{E_{\mathrm{err}}}_F^2.$$
Since
$\lambda_K^2\leq\left(1-\frac{1}{r}\right)^2n^2$ and
$s^+>\left(1-\frac{1}{r}\right)^2n^2$, we have
\begin{align*}
 0
 &<s^+-\left(1-\frac{1}{r}\right)^2n^2\\
 &\leq2\ip{(A_K)_+}{E_{\mathrm{err}}}+\norm {E_{\mathrm{err}}}_F^2.
\end{align*}
By \eqref{eq:ABplus_uniform}, the contribution to the inner product is
$2(\left(1-\frac{1}{r}\right)+o(1))$ for every internal edge and
$-2(\left(1-\frac{1}{r}\right)+o(1))$ for every missing crossing edge.  The error is
$o(1)(\ell+c)$, not merely $o(n^2)$.  Using \eqref{eq:E_norm}, we obtain
\begin{equation}
 0<\left(4\left(1-\frac{1}{r}\right)+2+o(1)\right)\ell-\left(4\left(1-\frac{1}{r}\right)-2-o(1)\right)c.
 \label{eq:c_over_ell}
\end{equation}
Since $r\ge3$, we have $4(1-1/r)-2>0$.
Consequently, $4(1-1/r)-2-o(1)>0$ for sufficiently
large $n$. If $\ell=0$, then \eqref{eq:c_over_ell} implies
\[
0<-\bigl(4(1-1/r)-2-o(1)\bigr)c\le0,
\]
a contradiction.
Thus we get,
\begin{equation*}
 \ell \geq1,
 \qquad
 c\leq
 \left(\frac{4\left(1-\frac{1}{r}\right)+2}{4\left(1-\frac{1}{r}\right)-2}+o(1)\right)\ell
 =\left(\frac{3r-2}{r-2}+o(1)\right)\ell. \qedhere
\end{equation*}
\end{proof}
Put
\[
 D=\max_v d_I(v).
\]
The maximality of the cut \eqref{eq:max_r_cut} implies
\begin{equation}
 d_G(v,V_j)\geq d_I(v)
 \quad\text{for every }v\in V_i\text{ and every }j\neq i.
 \label{eq:local_maxcut}
\end{equation}
Indeed, moving $v$ from $V_i$ to $V_j$ gains $d_I(v)$ crossing edges
and loses $d_G(v,V_j)$ crossing edges.

We next show that even the largest number of neighbors a vertex has within its own part is negligible compared with $n$.
\begin{claim} \label{cl:d_small}
\begin{equation}
 D=o(n).
\end{equation}
\end{claim}
\begin{proof}
Suppose not.  Passing to a subsequence, there are a fixed $\eta>0$, an
index $i$, and a vertex $v\in V_i$ such that
\[
 d_I(v)=D\geq\eta n.
\] 
By Claim~\ref{cl:cn_il},
\[
 \sum_u d_M(u)=2c=o(n^2).
\]
Define a sequence $\gamma_n$ as: when $c>0$, take
$\gamma_n=\sqrt{c/n^2}$; when $c=0$, take $\gamma_n=n^{-1/2}$.
Then $\gamma_n > 0$ and since $c=o(n^2)$, the sequence $\gamma_n\to0$. Define $\xi_n=\{u \in V(G): d_M(u) > \gamma_n n\}$.
When $c>0$, summing the missing degrees gives
\[
|\xi_n|
\le \frac{\sum_u d_M(u)}{\gamma_n n}
=\frac{2\gamma_n^2n^2}{\gamma_n n}
=2\gamma_n n=o(n).
\]
When $c=0$, $\xi_n$ is empty. Thus, all but $o(n)$ vertices satisfy $d_M(u)\leq\gamma_nn$.   

Since $D\geq \eta n$, there are $D-o(n)=\Omega(n)$ neighbors $u\in V_i$ of $v$ satisfying
$d_M(u)\leq \gamma_n n=o(n)$.  Fix one such $u$.  By \eqref{eq:local_maxcut}, for every
$j\neq i$,
\[
 |N_G(v)\cap V_j|\geq D.
\]
Since $u$ misses at most $d_M(u)$ vertices over all the other parts,
\begin{equation}
 |N_G(u)\cap N_G(v)\cap V_j|=|N_G(v)\cap V_j|-|(N_G(v)\cap V_j)\setminus N_G(u)|
 \geq D-d_M(u)\geq\eta n-o(n).
 \label{eq:common_in_each_part}
\end{equation}
Choose one vertex from each of the $r-1$ sets in
\eqref{eq:common_in_each_part}.  There are at least
$(\eta n-o(n))^{r-1}$ initial choices.  A fixed missing crossing edge
between two chosen vertices belongs to at most $n^{r-3}$ such choices.
The number of missing crossing edges in $G$ is $c$. Thus the number of choices missing at least one crossing edge is at most $c n^{r-3}$.
Thus at least
\[
 (\eta n-o(n))^{r-1}-cn^{r-3}=\Omega(n^{r-1})
\]
choices have every required crossing edge.  Each such choice, together with $u$ and $v$, is a $K_{r+1}$.  The families obtained from different internal neighbors $u$ are disjoint, because $uv$ is the unique internal edge of the resulting clique. Therefore
\[
 k_{r+1}=\Omega(n)\Omega(n^{r-1})=\Omega(n^r),
\]
contrary to \eqref{eq:se_contra_count}.  This proves
the claim. 

\end{proof}
We now count canonical candidates.  Such a candidate consists of
an internal edge in one part, which we call the distinguished part, and one vertex in every other part.  Using \eqref{eq:max_part_sizes}, the
number of candidates before checking the crossing edges is
\begin{align}
 N_0
 &=\sum_{i=1}^r |E(G[V_i])|\prod_{j\neq i}n_j\notag\\
 &=(1+o(1))\ell \left(\frac nr\right)^{r-1}.
 \label{eq:N0}
\end{align}

Fix a missing crossing edge $xy$, where $x\in V_a$ and $y\in V_b$.
If the distinguished part of a candidate is $V_a$, its internal edge must join
$x$ to an internal neighbor, so there are at most $Dn^{r-2}$ such
candidates containing $xy$.  The same bound holds when the distinguished part
is $V_b$.  If the distinguished part is $V_i$ with
$i\notin\{a,b\}$, there are at most $|E(G[V_i])|$ choices for the internal
edge and at most $n^{r-3}$ choices for the remaining vertices.  Hence a
fixed missing crossing edge destroys at most
$
 2Dn^{r-2}+\ell n^{r-3}
$
canonical candidates.

Every candidate which is not a clique contains a missing crossing edge.
By the union bound, Claim~\ref{cl:ell_c_bounds}, \eqref{eq:c_small}, and
Claim~\ref{cl:d_small}, the total number of destroyed candidates is at most
\begin{align*}
 c\bigl(2Dn^{r-2}+\ell n^{r-3}\bigr)
 &=\ell n^{r-1}\left(
 2\frac{c}{\ell}\frac Dn+\frac c{n^2}\right)\\
 &=o(\ell n^{r-1}).
\end{align*}
Every surviving candidate is a $K_{r+1}$, and it is counted exactly
once: its distinguished part and its unique internal edge are determined by
the clique.  Equation \eqref{eq:N0} therefore gives
\begin{align}
 k_{r+1}
 &\geq(1-o(1))\ell \left(\frac nr\right)^{r-1}\notag\\
 &\geq\left(\frac1{r^{r-1}}-o(1)\right)n^{r-1},
 \label{eq:cont}
\end{align}
where the last line uses $\ell \geq1$ from Claim~\ref{cl:ell_c_bounds}.  For sufficiently large $n$,
\eqref{eq:cont} contradicts \eqref{eq:se_contra_count}.  This proves
Theorem \ref{thm:main_square_energy} for every fixed $r\geq3$.

\section{The sharp triangle count: \texorpdfstring{$r=2$}{r = 2}}\label{sec:se_triangles}

We retain the maximum cut from \eqref{eq:max_r_cut}, but write it as
\[
 V(G)=V_1\sqcup V_2,
 \qquad a=|V_1|,\quad b=|V_2|.
\]
Thus
\[
 \ell=|E(G[V_1])|+|E(G[V_2])|,
 \qquad
 c=ab-e_G(V_1,V_2).
\]
Put
\[
 \delta=\frac{n^2}{4}-ab=\frac{(a-b)^2}{4}.
\]
Claim~\ref{cl:cn_il} gives
\begin{equation}
 a,b=\frac n2+o(n),
 \qquad
 \ell,c,\delta=o(n^2).
 \label{eq:ablhd}
\end{equation}

Let $A_{12}$ be the $a\times b$ biadjacency matrix of the bipartite graph consisting of the crossing edges, with rows indexed by $V_1$ and columns indexed by $V_2$, and put
\[
 m_c=\norm {A_{12}}_F^2=e_G(V_1,V_2)=ab-c.
\]
Let $A_1$ and $A_2$ be the adjacency matrices of $G[V_1]$ and $G[V_2]$,
respectively, and define
\[
 A_0=\begin{pmatrix}0&A_{12}\\ A_{12}^{\mathsf T}&0\end{pmatrix},
 \qquad
 A_{\mathrm{int}}=\begin{pmatrix}A_1&0\\0&A_2\end{pmatrix}.
\]
Then $A(G)=A_0+A_{\mathrm{int}}$,
\begin{equation}
 s^+(A_0)=m_c,
 \qquad
 \norm {A_{\mathrm{int}}}_F^2=2\ell.
 \label{eq:psiEF}
\end{equation}
The first identity follows because the nonzero eigenvalues
of $A_0$ are the numbers $\pm\sigma_i(A_{12})$ corresponding
to the positive singular values of $A_{12}$, with their multiplicities.

\begin{claim}\label{cl:sigma_splus}
    Let $\sigma=\sigma_1(A_{12})$. Let $k_3^\times$ be the number of triangles of $G$ which meet both $V_1$ and $V_2$. For sufficiently large $n$,
    $$\sigma \geq \frac{m_c}{\sqrt{ab}}>0$$ and $$ s^+ \leq m_c+2\ell+\frac{2k_3^\times}{\sigma}+2\sqrt{\ell c}.$$
\end{claim}
\begin{proof}
Since $m_c=n^2/4-o(n^2)$ by
\eqref{eq:ablhd}, we have $\sigma>0$ for large $n$. Since $\max\{t,0\}=\frac{t+|t|}{2}$, using spectral decomposition, we get
\begin{align*}
 (A_0)_+&=\frac12\left(A_0+(A_0^2)^{1/2}\right)  \notag \\
 &=\frac12\left[\begin{pmatrix}
  0&A_{12}\\
  A_{12}^{\mathsf T}&0
 \end{pmatrix}+ \begin{pmatrix}
  (A_{12}A_{12}^{\mathsf T})^{1/2}&0\\
  0&(A_{12}^{\mathsf T}A_{12})^{1/2}
 \end{pmatrix} \right]  \notag \\
 &=\frac12
 \begin{pmatrix}
  (A_{12}A_{12}^{\mathsf T})^{1/2}&A_{12}\\
  A_{12}^{\mathsf T}&(A_{12}^{\mathsf T}A_{12})^{1/2}
 \end{pmatrix}.
\end{align*}
Write
\begin{equation*}
 (A_{12}A_{12}^{\mathsf T})^{1/2}=\frac{A_{12}A_{12}^{\mathsf T}}\sigma+R_1,
 \qquad
 (A_{12}^{\mathsf T}A_{12})^{1/2}=\frac{A_{12}^{\mathsf T}A_{12}}\sigma+R_2.
\end{equation*}
If the singular values of $A_{12}$ are
$\sigma=\sigma_1\geq\sigma_2\geq\cdots$, then the eigenvalues of $R_1$ corresponding to the left singular vectors are
\[
 \sigma_i-\frac{\sigma_i^2}{\sigma}
 =\sigma_i\left(1-\frac{\sigma_i}{\sigma}\right).
\]
The same calculations work for $R_2$ corresponding to right singular vectors. The term for $i=1$ is zero and
$|1-\sigma_i/\sigma|\leq1$. Then
\begin{equation}
 \norm{R_1}_F^2\leq \sum \limits_{i=2}^{\min \{a,b\}}\sigma_i^2=tr(A_{12}A_{12}^T)-\sigma^2=m_c-\sigma^2,
 \qquad
 \norm{R_2}_F^2\leq m_c-\sigma^2.
 \label{eq:R_1R_2_norm}
\end{equation}

The $k_3^\times$ triangles that meet both $V_1$ and $V_2$ have a unique internal edge.  Since an off-diagonal
entry of $A_{12}A_{12}^{\mathsf T}$ counts common neighbors in $V_2$, and similarly
for $A_{12}^{\mathsf T}A_{12}$, we have
\begin{equation*}
 \ip{A_{12}A_{12}^{\mathsf T}}{A_1}
 +\ip{A_{12}^{\mathsf T}A_{12}}{A_2}=2k_3^\times.
\end{equation*}
Furthermore, Cauchy--Schwarz and \eqref{eq:R_1R_2_norm} give
\begin{align*}
 \ip{R_1}{A_1}+\ip{R_2}{A_2}
 &\leq
 \sqrt{\bigl(\norm{R_1}_F^2+\norm{R_2}_F^2\bigr)
       \bigl(\norm {A_1}_F^2+\norm {A_2}_F^2\bigr)}\notag\\
 &\leq2\sqrt{\ell (m_c-\sigma^2)}.
\end{align*}
Using the normalized all-ones vectors and
Cauchy--Schwarz, we obtain
$$ \frac{m_c}{\sqrt{ab}}=\frac{\one_a^{\mathsf T}A_{12}\one_b}{\sqrt{ab}} \leq \norm{\frac{\one_a}{\sqrt{a}}}\norm{\frac{A_{12}\one_b}{\sqrt{b}}}\leq \sigma.$$
Consequently,
\begin{equation}
 m_c-\sigma^2
 \leq m_c-\frac{m_c^2}{ab}
 =\frac{m_c c}{ab}\leq c.
 \label{eq:m_c}
\end{equation}

Apply \eqref{eq:Frob_2} to $A_0+A_{\mathrm{int}}$.  Equations
\eqref{eq:psiEF}--\eqref{eq:m_c} yield
\begin{align*}
 s^+
 &\leq m_c+2\ip{(A_0)_+}{A_{\mathrm{int}}}+2\ell \notag\\
 &=m_c+\ip{(A_{12}A_{12}^{\mathsf T)^{1/2}}}{A_1}+\ip{(A_{12}^{\mathsf T}A_{12})^{1/2}}{A_2}+2\ell \notag\\
 &=m_c+\frac{1}{\sigma}\left(\ip{A_{12}A_{12}^{\mathsf T}}{A_1}+\ip{A_{12}^{\mathsf T}A_{12}}{A_2}\right)+\ip{R_1}{A_1}+\ip{R_2}{A_2}+2\ell \notag\\
 &=m_c+2\ell+\frac{2k_3^\times}{\sigma}
   +\ip{R_1}{A_1}+\ip{R_2}{A_2}\notag\\
 &\leq m_c+2\ell+\frac{2k_3^\times}{\sigma}+2\sqrt{\ell c}. \qedhere
\end{align*} 
\end{proof}

\begin{claim}\label{cl:lcd}
    We have $\ell \geq 1$ and $ c+\delta=O_\varepsilon(\ell+1).$
\end{claim}
\begin{proof}
Under the contradiction assumption \eqref{eq:se_contra_count},
\[
 k_3^\times\leq k_3\leq\left(\frac12-\varepsilon\right)n=O(n).
\]
Claim~\ref{cl:sigma_splus} gives $\sigma\ge m_c/\sqrt{ab}$.
Since $m_c=n^2/4-o(n^2)$ by \eqref{eq:ablhd} and
$\sqrt{ab}\le n/2$, we obtain
\[
 \sigma\geq\frac{m_c}{\sqrt{ab}}\geq \frac{2m_c}{n}=\frac{n}{2}-o(n)=\left(\frac12-o(1)\right)n.
\]
Thus $2k_3^\times/\sigma=O_\varepsilon(1)$.  Since
\[
 m_c=ab-c=\frac{n^2}{4}-\delta-c
\]
and $s^+>n^2/4$, Claim~\ref{cl:sigma_splus} implies
\begin{equation*}
 \delta+c<2\ell+2\sqrt{\ell c}+O_\varepsilon(1).
\end{equation*}
Using $2\sqrt{\ell c}\leq c/2+2\ell $, we obtain
\begin{align*} 
c+\delta &< 4\ell+c/2+O_\varepsilon(1) \\
\delta+c/2 &< 4\ell+O_\varepsilon(1) \\
\delta+c \leq 2\delta+c &< 8\ell+O_\varepsilon(1). 
\end{align*}
Choosing a sufficiently large constant, which depends on $\varepsilon$, we obtain
\begin{equation*}
 c+\delta=O_\varepsilon(\ell+1).
\end{equation*}
Moreover $\ell \geq 1$.
Indeed, if $\ell=0$, then $G$ is bipartite, so its eigenvalues occur in
opposite pairs and
\begin{equation*}
 s^+=m=m_c\leq ab\leq\frac{n^2}{4},
\end{equation*}
contrary to \eqref{eq:se_strict_threshold}.
\end{proof}

\begin{claim}\label{cl:lcd:bigO}
    We have 
    \begin{equation}
 \ell=O_\varepsilon(1),
 \label{eq:ell_bigO}
\end{equation}
and 
\begin{equation*}
 c,\delta=O_\varepsilon(1).
\end{equation*}
\end{claim}
\begin{proof}
Recall that
\begin{equation}
 \sum_vd_M(v)=2c.
 \label{eq:sum_d_M}
\end{equation}
Fix a constant $\theta$ with $0<\theta<1/4$.  Call an internal edge $uv$
\emph{light} if
\[
 d_M(u)+d_M(v)\leq\theta n.
\]
Every light internal edge has at least
\[
 \min\{a,b\}-\theta n=\frac{a+b-|a-b|}{2}-\theta n=\frac n2 -\sqrt{\delta}-\theta n=\Omega_\theta(n)   \quad (\text{since } \delta=o(n^2) )
\]
common neighbors in the other part. Thus, for sufficiently large $n$, a light internal edge can be extended to at least $n/4$ triangles meeting both $V_1$ and $V_2$. Since every such triangle has a unique internal edge, the presence of another light internal edge would give at least another $n/4$ distinct triangles in $G$. Thus in that case $G$ would have at least $n/2$ triangles, contradicting $k_3(G) \leq \left( 1/2 -\varepsilon \right)n< n/2$  by \eqref{eq:se_contra_count}. Thus there is at most one light internal edge.

Every non-light internal edge has an endpoint in
\[
 \mathcal S=\{v:d_M(v)>\theta n/2\}.
\]
Equation \eqref{eq:sum_d_M} gives
$$\frac{\theta n}{2}|\mathcal{S}|\leq \sum \limits_{v \in \mathcal{S}} d_M(v) \leq \sum \limits_{v \in V(G)} d_M(v) =2c.$$
By Claim~\ref{cl:lcd} and $\delta\ge0$, we have
$c=O_\varepsilon(\ell+1)$. Hence
\begin{equation}
 |\mathcal S|
 \leq\frac{4c}{\theta n}
 =O_\varepsilon\left(\frac{\ell+1}{n}\right).
 \label{eq:S_size}
\end{equation}
Since there are at least $\ell-1$ non-light internal edges and each of them has an endpoint in $\mathcal S$, we obtain
\begin{equation}
 \ell-1\leq\sum_{v\in\mathcal S}d_I(v).
 \label{eq:S_degree_sum}
\end{equation}

We claim that $\ell=O_{\varepsilon}(1)$. Suppose on the contrary that $\ell$ is not $O_{\varepsilon}(1)$. Passing to a subsequence we may assume that $\ell \rightarrow \infty$. This forces that eventually, $\ell \geq 2$ and hence by \eqref{eq:S_degree_sum}, $\mathcal S$ is nonempty for sufficiently large $n$ along this subsequence. Thus equations
\eqref{eq:S_size}--\eqref{eq:S_degree_sum} yield a vertex
$v\in\mathcal S$ such that
\begin{equation}
 d_I(v)\geq
 \frac{\ell-1}{|\mathcal S|}=\Omega_\varepsilon(n).
 \label{eq:Ideg_lower}
\end{equation}
Since $V_1\sqcup V_2$ is a maximum cut, the crossing degree of $v$ satisfies
\begin{equation}
 d_c(v)\geq d_I(v).
 \label{eq:cdeg_lower}
\end{equation}

Every edge between $N_I(v)$ and $N_c(v)$ gives a triangle containing $v$ and all the triangles formed in this way are different. Therefore using \eqref{eq:cdeg_lower},
\begin{align*}
k_3 &\geq e(N_I(v), N_{c}(v))  \\
  &\geq |N_I(v)| ~ |N_{c}(v)|-c \\
  &=d_I(v) d_{c}(v)-c \\
  &\geq d_I(v)^2 -c.
\end{align*}
Using \eqref{eq:Ideg_lower} and \eqref{eq:c_small} we get $k_3 = \Omega_{\varepsilon}(n^2)$,
contrary to $k_3=O(n)$.  Thus \eqref{eq:ell_bigO} holds.
 
Note that, since $c, \delta \geq 0$, Claim~\ref{cl:lcd} and \eqref{eq:ell_bigO} give
\begin{equation*}
 c,\delta=O_\varepsilon(1). \qedhere
\end{equation*}\end{proof}
Choose any internal edge $uv$, which exists by
Claim~\ref{cl:lcd}.  The sets of missing crossing edges incident
with $u$ and $v$ are disjoint as edge sets and together contain
$d_M(u)+d_M(v)$ edges.  Hence $d_M(u)+d_M(v)\leq c$.  The edge $uv$ therefore
has at least
\[
\min \{a,b\}- (d_M(u)+d_M(v)) \geq \min\{a,b\}-c
\]
common neighbors in the other part. By Claim~\ref{cl:lcd:bigO}, both $\delta$ and $c$ are
$O_\varepsilon(1)$. Therefore
$$\min\{a,b\}
=\frac n2-\sqrt{\delta}
=\frac n2-O_\varepsilon(1),$$
and consequently
$$
k_3\ge \min\{a,b\}-c
\ge \frac n2-O_\varepsilon(1).
$$
For sufficiently large $n$, this contradicts
$k_3\leq(1/2-\varepsilon)n$.  This proves
Theorem \ref{thm:main_square_energy} for $r=2$.

\section{Sharpness of the coefficient}\label{sec:se_sharpness}

For each integer $\nu \geq 2$, let $n=r\nu$, start with the balanced Tur\'{a}n graph $T_r(n)$, and add one
edge inside one of its parts.  The resulting graph has
\[
 m=t_r(n)+1=\frac{\left(1-\frac{1}{r}\right)n^2}{2}+1.
\]
By the Rayleigh quotient of the all-ones vector,
\[
 \lambda_1\geq\frac{2m}{n}=\left(1-\frac{1}{r}\right)n+\frac2n.
\]
Since $\sqrt{s^+}\geq\lambda_1$, this graph satisfies
\[
 \sqrt{s^+}>\left(1-\frac{1}{r}\right)n.
\]
Every copy of $K_{r+1}$ must use the added edge and one vertex from each
of the other $r-1$ parts.  Therefore
\[
 k_{r+1}=\nu^{r-1}=\frac{n^{r-1}}{r^{r-1}}.
\]
Thus no coefficient larger than $r^{-(r-1)}$ can replace the coefficient
in Theorem \ref{thm:main_square_energy}.  This completes the proof of
Theorem \ref{thm:main_square_energy}.

\section*{Acknowledgments}
Abhay Jayarajan acknowledges support from the DST INSPIRE Fellowship (IF220692). M.~Rajesh Kannan acknowledges financial support from ANRF-CRG, India (File No.\ CRG/2023/002747). 	

\section*{Declaration of AI use}
The authors acknowledge the use of ChatGPT (GPT-6 Astra, OpenAI; accessed September 2026) for improving the exposition of the manuscript and refining some of the proofs. All mathematical statements have been independently verified by the authors, who are solely responsible for the correctness of the results presented in this paper.

\section*{Data Availability:} No data were generated or analysed in this study, as the work is entirely theoretical. 

\section*{Conflict of interest:} 
The authors state that they possess no conflicts of interest.

\begingroup
\small
\setlength{\parskip}{3pt}
\bibliographystyle{plain}
\bibliography{P1}
\endgroup

\vspace{0.4cm}

\affl{Abhay Jayarajan}{ma23resch02001@iith.ac.in, abhayjayarajan@gmail.com}{ Department of Mathematics, Indian Institute of Technology Hyderabad, Kandi, Sangareddy 502284, India.}

\affl{M. Rajesh Kannan}{rajeshkannan@math.iith.ac.in, rajeshkannan1.m@gmail.com}{Department of Mathematics, Indian Institute of Technology Hyderabad, Sangareddy 502285, India}

\affl{Shivaramakrishna Pragada}{shivaramakrishna\_pragada@sfu.ca, shivaramkratos@gmail.com}{Department of Mathematics, Simon Fraser University, Burnaby, Canada}

\end{document}